\documentclass{article}
\usepackage{fullpage}

\usepackage{theorem,graphicx,amssymb,amsmath,fancyhdr,todonotes}

\theorembodyfont{\slshape}

\newtheorem{theorem}{Theorem}
\newtheorem{lemma}[theorem]{Lemma}

\newtheorem{prop}[theorem]{Proposition}

\newtheorem{conj}[theorem]{Conjecture}

\def\QED{\ensuremath{{\square}}}
\def\markatright#1{\leavevmode\unskip\nobreak\quad\hspace*{\fill}{#1}}
\newenvironment{proof}
 {\begin{trivlist}\item[\hskip\labelsep{\bf Proof.}]}
 {\markatright{\QED}\end{trivlist}}

\usepackage{microtype} 

\newcommand\blfootnote[1]{%
	\begingroup
	\renewcommand\thefootnote{}\footnote{#1}
	\addtocounter{footnote}{-1}%
	\endgroup
}

\usepackage[hyperfootnotes=false]{hyperref}

\usepackage{tikz}

\usepackage{bm}
\usepackage{comment}

\newcommand{\rcrs}{\overline{\operatorname{cr}}}
\newcommand{\crs}{\operatorname{cr}}

\usepackage{thm-restate}

\title{On the rectilinear crossing number of complete balanced multipartite graphs and layered graphs \footnote{A preliminary version of this work has been presented at EGC'23\cite{egc}}}
\author{Ruy Fabila-Monroy\thanks{Departamento de Matem\'{a}ticas, CINVESTAV} \thanks{{\tt ruyfabila@math.cinvestav.edu.mx}}  
\and Rosna Paul \thanks{Institute for Software Technology, Graz University of Technology, Graz, Austria} \thanks{Supported by the Austrian Science Fund (FWF) grant W1230.}  \thanks{\tt{ropaul@ist.tugraz.at}} 
\and Jenifer Viafara-Chanchi \footnotemark[2] \thanks{\tt{viafara@math.cinvestav.mx }}
\and Alexandra Weinberger\footnotemark[4] \footnotemark[5] \thanks{Department of Mathematics and Computer Science, FernUniversit\"at in Hagen, Hagen, Germany} \thanks{\tt{alexandra.weinberger@fernuni-hagen.de}} }

\begin{document}
\maketitle

\blfootnote{{
		\begin{minipage}[l]{.85\textwidth}\hspace{-1pt}
			This project has received funding from the European Union's Horizon 2020 research and innovation programme under the Marie Sk\l{}odowska-Curie grant agreement No 734922.
		\end{minipage}\hspace{4.5pt}
		\begin{minipage}[l][1.2cm]{0.08\textwidth} \vspace{-7pt}
			\includegraphics[trim=10cm 6cm 10cm 5cm,clip,scale=0.13]{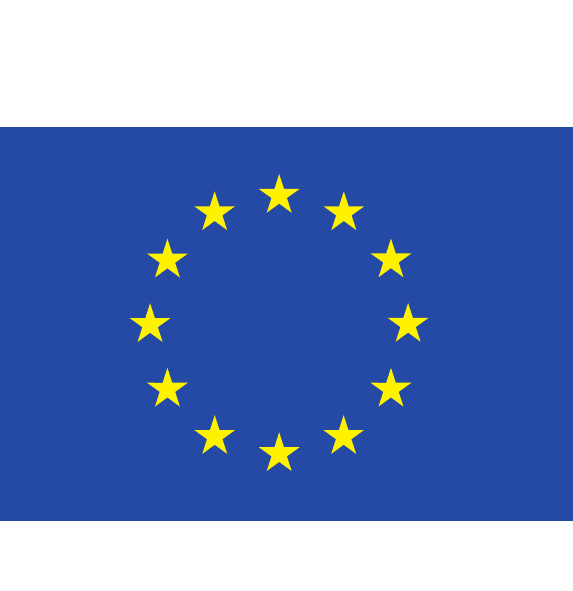} 
		\end{minipage}
		}}
\begin{abstract}
A rectilinear drawing of a graph is a drawing of the graph in the plane
in which the edges are drawn as straight-line segments.
The rectilinear crossing number of a graph is the minimum number
of pairs of edges that cross over all rectilinear drawings of the graph.
Let $n \ge r$ be positive integers.
The graph $K_n^r$,  is the complete $r$-partite
graph on $n$ vertices, in which every set of the partition has at least $\lfloor n/r \rfloor$ vertices.
The layered graph, $L_n^r$, is an $r$-partite graph on $n$ vertices, where $n$ is multiple of $r$. Every partition
of $L_n^r$ contains $n/r$ vertices; for every $1\le i \le r-1$,
all the  vertices in the $i$-th partition are adjacent to all the vertices in the $(i+1)$-th
partition, and these are the only edges of $L_n^r$. In this paper, we give upper bounds
on the rectilinear crossing numbers of $K_n^r$ and~$L_n^r$.%
\end{abstract}

\section{Introduction}
 Let $G$ be a graph on $n$ vertices 
 and let $D$ be a drawing of $G$. The \emph{crossing number} of $D$ is the number, $\crs(D)$, of pairs of edges that cross in~$D$. 
 The \emph{crossing number} of $G$ is the minimum crossing number, $\crs(G)$, over all drawings of $G$ in the plane.
 A \emph{rectilinear drawing} of $G$ is a drawing of $G$ in the plane
 in which its vertices are points in general position, and its
 edges are drawn as straight-line segments joining these points. 
 The \emph{rectilinear crossing number} of $G$, 
 is the minimum crossing number, $\rcrs(G)$, 
 over all rectilinear drawings of $G$ in the plane.
Computing crossing and rectilinear crossing numbers of graphs are important problems in Graph Theory and Combinatorial Geometry. 
For a comprehensive review of the literature on crossing numbers, we refer the reader to Schaefer's book~\cite{schaefer2018crossing}. 

Most of the research on crossing numbers have been focused around the complete graph, $K_n$, and the complete bipartite graph $K_{m,n}$.
For the complete graph, Hill~\cite{hararynumber} gave the following drawing of $K_n$; see Figure~\ref{fig:hill}~(left) for an example. 
Place half of the vertices
equidistantly on the top circle of a cylinder, and the other half equidistantly on the bottom circle.
Join the vertices with geodesics on the cylinder. 
Hill showed that the following number, $H(n)$, is the crossing number of this drawing, and it is now conjectured to be optimal.
Let
\[H(n) :=\frac{1}{4} \left \lfloor \frac{n}{2} \right \rfloor \left \lfloor \frac{n-1}{2} \right \rfloor \left \lfloor \frac{n-2}{2} \right \rfloor \left \lfloor \frac{n-3}{2} \right \rfloor.\]

\begin{conj}[Harary-Hill~\cite{guy_conj}]
\[\crs(K_n) =H(n).\]
\end{conj}

\begin{figure}
	\centering
	\includegraphics[page=2]{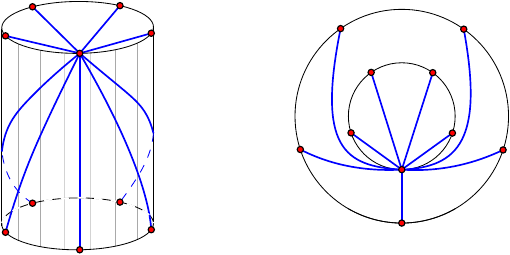}
	\caption{An example of Hill's drawings of $K_{10}$, where here for convenience only the edges of one vertex are drawn. Left: the drawing on a cylinder. Right: an equivalent representation of Hill's drawings via concentric circles.}
	\label{fig:hill}
\end{figure}
Let
\[Z(m,n):=\left \lfloor \frac{n}{2} \right \rfloor \left \lfloor \frac{n-1}{2} \right \rfloor \left \lfloor \frac{m}{2} \right \rfloor \left \lfloor \frac{m-1}{2} \right \rfloor\]
and 
\[Z(n):=Z(n,n).\]
Zarankiewicz~\cite{zaran} gave a drawing of the complete bipartite graph $K_{m,n}$ with $Z(m,n)$ crossings, which he claimed to be optimal. 
Kainen and Ringel  independently found a flaw in Zarankiewicz proof (see~\cite{Guy}). 
\begin{conj}[Zarankiewicz]
\[\crs(K_{m,n}) = Z(m,n).\]
\end{conj}
It is widely conjectured that Zarankiewicz conjecture holds.
Zarankiewicz drawing of $K_{m,n}$ is rectilinear; thus we also have the following.
\begin{conj}
\[\rcrs(K_{m,n})=\crs(K_{m,n}).\]
\end{conj}

Much less is known for the rectilinear crossing number of the complete graph.
\begin{prop}\label{prop:crs<rcrs}
 For $n \ge 10$, \[ \crs(K_{n}) <  \rcrs(K_n).\]
\end{prop}
This result seems to be folklore; for completenes we provide a proof in the appendix. 
In contrast to the case of the complete bipartite graph, there is no conjectured
value for $\rcrs(K_n)$, nor drawings conjectured to be optimal. The best bounds to date are
\[ 0.379972 \binom{n}{4} <\rcrs(K_n) <  0.380445 \binom{n}{4}+O(n^3).\]
The lower bound is due to \'{A}brego, Fern\'{a}ndez-Merchant, Lea\~{n}os, and Salazar~\cite{aproach}, and the upper bound to
Aichholzer, Duque,  Fabila-Monroy, Garc{\'i}a-Quintero, and Hidalgo-Toscano~\cite{upper_new}.
It is known that \[\lim_{n \to \infty}\frac{\rcrs(K_n)}{\binom{n}{4}}=\overline{q}, \]
for some positive constant $\overline{q}$; this constant is known as the \emph{rectilinear
crossing constant}. For a proof of this fact see the paper by Scheinerman and Wilf~\cite{wilf}.

Let $K_{n_1, n_2,\dots, n_r}$ be the complete $r$-partite graph with $n_i$ vertices
in the $i$-th set of the partition;
and let $K_{n}^r$ be the complete balanced $r$-partite graph in which there are at least 
$\lfloor n/r \rfloor$ vertices in every partition set. 
Harborth~\cite{harborth_rp} gave a drawing that provides an upper bound for 
$\crs (K_{n_1, n_2,\dots,n_r})$; and gave an explicit formula for this number, which he conjectured to be optimal. He observed that for the case of $r=3$,
his drawing can be made rectilinear.
More recently,  Gethner, Hogben, Lidick\'y, Pfender, Ruiz and Young~\cite{gethner2017crossing} independently studied
the problem of the crossing number and rectilinear crossing number of complete balanced $r$-partite graphs. For $r=3$, they obtain
the same bound as Harborth; and their drawing is the rectilinear version of Harborth's drawing.

Let $r$ be a positive integer and let $n$ be a multiple
of $r$. The \emph{balanced layered graph}, $L_n^r$, is the graph defined as follows. Its
vertex set  is partitioned into sets $V_1,\dots,V_r$, each consisting of $n/r$ vertices.
We call the set $V_i$, the $i$-th layer of $L_n^r$. The edge set of $L_n^r$ is given by
\[\{uv: u \in V_i \textrm{ and } v \in V_{i+1}, \textrm{ for } i=1,\dots,r-1\};\]
that is, the edges are exactly all possible edges between vertices on consecutive layers.

In this paper, we mainly focus on the rectilinear crossing numbers of $K_n^r$ and $L_n^r$ .
If $n$ is fixed and $r$ tends to $n$,
then $K_n^r$ tends to $K_n$. We believe that studying the rectilinear
crossing number of $K_n^r$ might shed some light on how optimal rectilinear drawings 
of $K_n$ look like. 

This paper is organized as follows. In Section~\ref{sec:random}, we give a general technique to obtain
non-rectilinear and rectilinear drawings of a given graph $G$ on $n$ vertices. It simply consists of mapping
randomly the vertices of $G$ to optimal drawings of $K_n$. We show how this technique upper bounds 
$\crs(K_n^r)$ and $\rcrs(K_n^r)$. The bounds obtained in this way are very close to being optimal.
However, for the layered graphs this technique gives rather poor upper bounds.
In Section~\ref{sec:planted}, we give a technique were given an specific drawing of a graph, we use
this drawing as a ``seed'' to produce larger drawings by replacing each vertex $u$ with a cluster of collinear vertices
$S_u$ arbitrarily close to $u$. In the new drawing two vertices in different clusters $S_u$ and $S_v$ are adjacent whenever $u$ and $v$ are adjacent in the original drawing.
We call the new larger drawing a ``planted drawing''.
The conjectured crossing optimal drawings of $K_{n,n}$ and $K_{n,n,n}$ mentioned above are actually
planted drawings with drawings of $K_{2,2}$ and $K_{2,2,2}$ as seeds, respectively. However, we show that there is no rectilinear drawing of $K_4$ or $K_8^4$ that can be the seed of a crossing optimal planted drawing of $K_n^4$. For the layered graph, we give a rectilinear planar drawing of $L_{2r}^r$. When used as a seed,
this drawing produces a planted drawing of $L_n^r$, with significantly smaller crossing number, than those produced by the random embedding technique.
The proofs of many of our results are long and technical; for the sake of clarity, we have relocated most of the proofs and constructions to an appendix.

\section{Random Embeddings into Drawings of $K_n$ with Small Crossing Number}\label{sec:random}

Suppose that we have a drawing (that can be rectilinear but does not have to be) $D'$ of $K_n$.
If $\crs(D')$
is small, it might be a good idea to use this drawing to produce a drawing of a graph $G$ on $n$ vertices.
Let $D$ be the drawing of $G$ that is produced by mapping the vertices of $G$ randomly 
to the vertices of $D'$,  and where the edges are drawn as their corresponding edges of $D'$.
We call $D$ a \emph{random embedding} of $G$ into $D'$.

In every $4$-tuple of vertices of $D'$, there are three pairs of independent edges, which could cross.
Of these three pairs at most one pair is crossing.
For every pair of independent edges of $G$, we have
a possible crossing in $D$; thus, the probability that this pair of edges is mapped
to a pair of crossing edges is equal to 
\[ \frac{1}{3}\cdot \frac{\crs(D')}{\binom{n}{4}}. \]
By defining, for every pair of independent edges of $G$, an indicator random variable
with value equal to one if the edges cross and zero otherwise, we obtain the following expression for the expected value of $\crs(D)$,  where $||G||$ is the number of edges in $G$ and $d(v)$ is the degree of a vertex~$v$ of $G$.
\begin{equation}
 \operatorname{E}(\crs(D)) = \frac{cr(D')}{3\binom{n}{4}} \left ( \binom{||G||}{2}-\sum_{v\in V(G)} \binom{d(v)}{2} \right ). \label{eq:random_embeding}
\end{equation}
%

\subsection*{Complete Balanced $r$-partite Graphs}

For an upper bound on the crossing number of $K_n^r$, we use Equation~\ref{eq:random_embeding} and Hill's drawing of $K_n$.
\begin{restatable}{theorem}{knrgen}\label{thm:Knr_gen}
Suppose that $n$ is a multiple of $r$. Let  $D$ be a random embedding of $K_n^r$ into Hill's drawing of~$K_n$.
Then,
\[\crs(K_n^r) \le E(\crs(D)) \le \frac{1}{16} \left (\frac{r-1}{r} \right )^2 \left(\frac{n^4}{4}-\frac{3n^3}{2} \right) +O(n^2).\]
\end{restatable}
In~\cite{gethner2017crossing}, the authors obtain the same bound on $\crs(K_n^r)$ by considering a random mapping of the vertices of $K_n^r$ into a sphere, and then joining the
corresponding vertices with geodesics. This type of drawing is called a \emph{random geodesic spherical drawing.} In 1965, Moon~\cite{moon65}, showed that the expected number of crossings of a random geodesic spherical
drawing of $K_n$ is equal to 
\[\frac{1}{16}\binom{n}{2}\binom{n-2}{2}  =H(n)-O(n^3);\]
which explains why the bound of Theorem~\ref{thm:Knr_gen} matches the bound of \cite{gethner2017crossing}.

The number, $H(n,r)$,  of crossings in Harborth's~\cite{harborth_rp} drawing of $K_{n}^r$, when $n$ is a multiple of $r$ is at most
\[H(n,r) \le \frac{3}{8} \binom{r}{4} \frac{n^4}{r^4}+
r\left \lfloor \frac{n/r}{2} \right \rfloor \left \lfloor \frac{n/r-1}{2} \right \rfloor \left \lfloor \frac{n-n/r}{2} \right \rfloor \left \lfloor \frac{n-n/r-1}{2} \right \rfloor
-\binom{r}{2}\left ( \left  \lfloor \frac{n/r}{2} \right \rfloor ^2\right)\left ( \left  \lfloor \frac{n/r-1}{2} \right \rfloor ^2\right)+O(n^2).\]
Due to 
the complexity of the formula, we use the following approximation to $H(n,r)$ instead. 
\begin{restatable}{restlemma}{harborth}\label{lem:Hnr}
    If $n$ is a multiple of $r$, then
    \[H(n,r) \le \frac{1}{16} \left ( \frac{r-1}{r}\right )^2 \left ( \frac{n^4}{4} -2 n^3 \right ) +O(n^2).\]
\end{restatable}  

Let $D$ be as in Theorem~\ref{thm:Knr_gen}; note that by Lemma~\ref{lem:Hnr}, it holds that
\[E(\crs(D))-H(n,r) \le \frac{1}{32}\left ( \frac{r-1}{r}\right )^2  n^3+O(n^2)=O(n^3).\]
Thus, the random embedding gives an upper bound on $\crs(D)$ that matches the conjectured value up to  the leading term, but it is a little worse
in the lower terms.


We now upper bound $\rcrs(K_n^r)$, with this technique.
\begin{restatable}{theorem}{upperrect}\label{lem:upper_rect}
Let $r$ be a positive integer and $n$ be a multiple of $r$.
Let $\overline{D}$ be a random embedding of $K_n^r$ into an optimal rectilinear drawing of $K_n$. Then
\begin{align*}
\rcrs (K_n^r)  & \le E(\crs(\overline{D})) \\
& \le \frac{\overline{q}}{4!} \left( \frac{r-1}{r} \right )^2 n^4+o(n^4)\\
& < 0.015852 \left( \frac{r-1}{r} \right )^2 n^4+o(n^4).
\end{align*}
\end{restatable}
For a lower bound we have the following. 

\begin{restatable}{theorem}{lowerrect}\label{lem:lower_rect}
Let $r$ be a positive integer and $n$ be a multiple of $r$.  
Then
    \[ \rcrs(K_n^r) \ge \rcrs(K_r)  \left ( \frac{n}{r} \right )^4.\]
\end{restatable}

Theorems~\ref{lem:upper_rect} and~\ref{lem:lower_rect} imply the following.
\begin{restatable}{restcor}{corolrestate}\label{cor:lime}
Let $r=r(n)$ be a monotone increasing function of $n$ such that $r \to \infty$ as $n \to \infty$.
 Then
\[\lim_{n \to \infty} \frac{\rcrs(K_n^r)}{\binom{n}{4}} = \overline {q}.\]
\end{restatable}

In both~\cite{harborth_rp} and~\cite{gethner2017crossing}, it is conjectured that 
\[\crs \left (K_n^3 \right )=\rcrs \left (K_n^3 \right ).\]
Using the order type database~\cite{otype}, we have verified that 
\[\rcrs \left ( K_8^4 \right )=8 
\textrm{ and }
\rcrs \left ( K_9^4 \right )=15. \]
On the other hand 
\[\crs \left (K_8^4 \right ) \le H(8,4)=6 \textrm{ and }
\crs \left (K_9^4 \right ) \le  H(9,4)=15.\]

See Figure~\ref{fig:Hnrn} for an example.
\begin{figure}[t]
	\centering
	\includegraphics[]{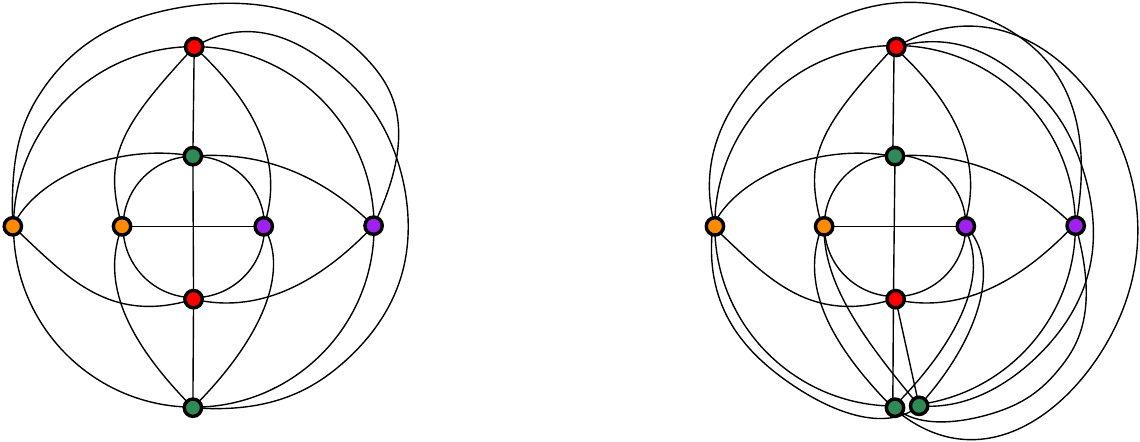}
	\caption{A drawing of $K_8^4$ with 6 crossings (left) and $K_9^4$ with 15 crossings (right).} 
	\label{fig:Hnrn}
\end{figure}
From the above results we conjecture the following.
\begin{conj}
There exists a natural number $n_0 > 9$ such that for all $n \ge n_0$,
\[\crs  \left (K_n^4 \right ) < \rcrs \left ( K_n^4 \right ).\]
\end{conj}
%

\subsection*{Layered Graphs}

Using the random embedding technique into Hill's drawing of~$K_n$, we obtain the following upper bound for~$\crs (L_n^r)$.
\begin{restatable}{theorem}{randomlayers}\label{thm:layers_random}
    \[\crs(L_n^r) \le \frac{(r-1)^2}{16 r^4}n^4+O(n^3).\]
\end{restatable}
We improve this upper bound in Section~\ref{sec:planted}.

\section{Planted Rectilinear Drawings}\label{sec:planted}

Let $D$ be a rectilinear drawing of a graph $G$. For every vertex $v$ of $D$, let
$\ell_v$, be a directed straight line passing through $v$ and no other vertex of $D$, such that
the left halplane of $\ell_v$ contains $\lfloor d(v)/2 \rfloor$ neighbors of $v$ and the right halfplane of $\ell_v$
contains the remaining $\lceil d(v)/2 \rceil$ neighbors of $v$.
Let $G^s$ be the graph whose vertex set is equal to
\[\{(v,i):i=1,\dots,s \textrm{ and } v \in V(G)\},\]
and in which $(v,i)$ is adjacent to $(w,j)$ whenever
$vw$ is an edge of $G$.  We say that the set
$\{(v,1),\dots,(v,s)\}$ is the \emph{cluster} of $v$.
Let $D^s$ be the rectilinear drawing of $G^s$ in which for every vertex $v$ of $G$, the vertices of  cluster
are placed arbitrarily close to $\ell_v$ and arbitrarily close to $v$ (in $D$).
We say that $D^s$ is a \emph{planted drawing} of $G^s$ with \emph{seed} $D$.

\begin{restatable}{restlemma}{lemseed}\label{lem:seed}
 \[\crs(D^s)=\crs(D)s^4+\sum_{v \in V(G)} \left ( \binom{ \lfloor d(v)/2 \rfloor}{2}+ \binom{\lceil d(v)/2 \rceil}{2}\right)\frac{s^3(s-1)}{2}+||G||\frac{s^2(s-1)^2}{4},\]
where we follow the standard convention that $\binom{n}{m}=0$ when $n < m$.
 \end{restatable}
Seeds and planted drawings\ were first used by Ábrego and Fernández-Merchant~\cite{bersil}\footnote{They do it in a different way as presented here; first they duplicate each vertex
along halving lines; then they choose halving lines for the original and new vertices and duplicate a new. They iterate this process.} to upper bound the rectilinear
crossing number of $K_n$. 
The current best upper bound on $\rcrs(K_n)$ is obtained via a seed of 2643 vertices and 771218714414 crossings.

\subsection*{Complete Balanced $r$-partite Graphs}
Note that if we use $K_{tr}^r$ as a seed for a planted drawing of $K_n^r$, we have that $s=\frac{n}{tr}$.
Thus, from Lemma~\ref{lem:seed}  we obtain the following.
\begin{restatable}{restcor}{corseed}\label{cor:seed}
Let $D$ be a rectilinear drawing of $K_{tr}^r$. Then using $D$ as a seed
 we obtain a planted drawing of $K_n^r$ with
 \[\left (\frac{\crs(D)+\frac{rt}{2}\left(\binom{\lfloor (r-1)t/2 \rfloor}{2}+\binom{\lceil (r-1)t/2 \rceil}{2}\right ) +\frac{r(r-1)t^2}{8}}{(rt)^4} \right )n^4-O(n^3) \]
 crossings.
\end{restatable}
Using the seeds in Figure~\ref{fig:seeds}, we obtain planted rectilinear drawings
of $K_n^2$ and $K_n^3$, with the conjectured minimum number of crossings.
\begin{figure}
\centering
\includegraphics[width=0.55 \linewidth]{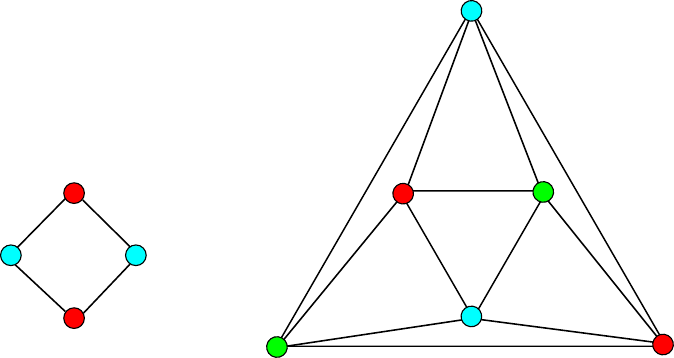}
\caption{The seeds for the planted drawings of $K_n^2$ and $K_n^3$}
\label{fig:seeds}
\end{figure}

Using the random embedding technique and Theorem~\ref{thm:Knr_gen} we obtain
a rectilinear drawing of $K_n^4$ with at most
\begin{equation}\label{eq:upper}
0.0089676n^4+o(n^4)
\end{equation}
crossings; and since $\overline{q} > 0.379972$, the best we can hope to achieve with
the random embedding technique is a rectilinear drawing of $K_n^4$ with
\begin{equation}\label{eq:lower}
 0.0089055n^4+o(n^4)
\end{equation}
crossings. 

Using a planar drawing of $K_4$ as a seed, we obtain a rectilinear planted
 drawing of $K_n^4$ (in this case $r=4$ and $t=1$) with 
 \[\left( \frac{2 \left ( \binom{1}{2}+\binom{2}{2}\right)+\frac{3}{2}}{4^4}\right)n^4-O(n^3)
  =\frac{7}{2^9}n^4-O(n^3) 
  =0.013671875 n^4-O(n^3)\]
 crossings.
 Using a rectilinear drawing of $K_8^4$ with $8$ crossings as a seed, we obtain a
 planted rectilinear drawing of $K_n^4$ with
 \[\left( \frac{8+4 \left( \binom{3}{2}+\binom{3}{2}\right)+6}{8^4}\right)n^4-O(n^3)=\frac{38}{8^4}n^4-O(n^3)=0.009277344n^4-O(n^3)\]
 crossings.
 
 Fabila-Monroy and L\'opez~\cite{smallsets} used an heuristic of randomly moving
 vertices to obtain a rectilinear drawing of $K_{75}$ with $45049$ crossings. This was
 used as a seed for a previous best upper bound on $\overline{q}$. In~\cite{crucenuevo} 
 Duque, Fabila-Monroy, Hern\'andez-V\'elez and Hidalgo-Toscano gave an $O(n^2 \log n)$ time algorithm
 to compute the crossing number of a rectilinear drawing of a graph on $n$ vertices. Using a similar heuristic
 as in~\cite{smallsets} and the algorithm of~\cite{crucenuevo}, we obtained
 a rectilinear drawing of $K_{24}^4$ with $2033$ crossings. Using this as a seed we obtain
 a planted rectilinear drawing of $K_n^4$ with
 \[ \left( \frac{2033+12\left(\binom{9}{2}+\binom{9}{2}\right)+54}{24^4}\right)n^4-O(n^3)
  =\frac{2951}{24^4}n^4-O(n^3)
  =0.0088946n^4-O(n^3)\]
 crossings. This is better than the best possible upper bound
 obtainable with the random embedding technique.
 However, for $r \ge 5$, we have not found seeds that provide planted drawings with less crossings than the drawings
 obtained from the random embedding technique.
 
\subsection*{Layered Graphs}

We now show a rectilinear planar drawing $D_r$ of $L_{2r}^r$.
For $i=1,\dots,r$, let $\{u_i,v_i\}$ be the two vertices on layer $i$ of $L_{2r}^r$.
Place $u_i$ and $v_i$ at the points $p_i$ and $q_i$, respectively; where
\[p_i:=
\begin{cases}
  (i,0) \textrm{ if } i \textrm{ is odd, } \\
 (0,i) \textrm{ if } i \textrm{ is even,}
\end{cases}
\textrm{ and } \ 
q_i:=
\begin{cases}
  (-i,0) \textrm{ if } i \textrm{ is odd, } \\
 (0,-i) \textrm{ if } i \textrm{ is even.}
\end{cases}\]
See Figure~\ref{fig:layer} for the drawing of $L_{12}^6$.
\begin{figure}
\centering
\includegraphics[width=0.35 \linewidth]{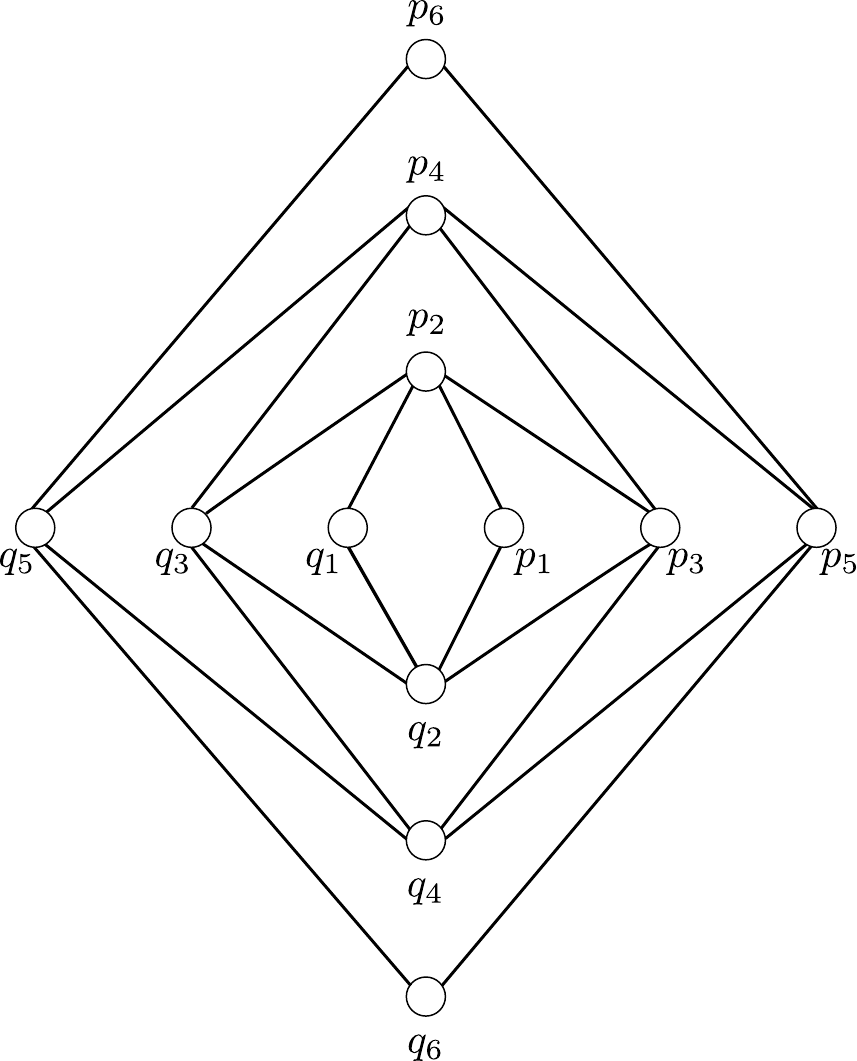}
\caption{The rectilinear $D_{6}$ drawing of $L_{12}^6$}
\label{fig:layer}
\end{figure}

Using this drawing as a seed for a planted drawing of $L_n^r$, we obtain
a rectilinear drawing with
\begin{align*}
 &\sum_{v \in V(D_r)} \left ( \binom{ \lfloor d(v)/2 \rfloor}{2}+ \binom{\lceil d(v)/2 \rceil}{2}\right)\frac{s^4}{2}+||D_r||\frac{s^4}{4}-O(s^3)\\
 &= (2(r-2)\cdot 2)\frac{n^4}{2\cdot (2r)^4}+4\cdot (r-1)\frac{n^4}{4\cdot (2r)^4}-O(n^3)\\
 &=\frac{3r-5}{16 r^4}n^4-O(n^3)
\end{align*}
crossings.
For $r \ge 4$, this is better than the upper bound obtained with the random embedding technique.

For $i=2,\dots,r-1$, let $H_i$ be the subgraph of $L_n^r$ induced by the vertices
in layers $i-1,i$ and $i+1$. Note that this graph is isomorphic to $K_{n/r,2n/r}$.
Thus, assuming that Zarankiewicz's conjecture holds, in every drawing of $L_n^r$, $H_i$
produces at least $Z(n/r,2n/r)$ crossings. Each of these crossings is produced
by at most two such $H_i$'s. Therefore, assuming that Zarankiewicz's conjecture is true, we have that
\[ \crs(L_n^r) \ge \frac{(r-2)}{2}Z(n/r,2n/r)=\frac{2r-4}{16r^4}n^4-O(n^3). \]



{\small
\bibliographystyle{abbrv}
\bibliography{crossing_number_beyond_Kn}}
\newpage

\onecolumn

\section{Appendix}
Let $D$ be a rectilinear drawing of $K_n$.
For $0 \le j \le n-2$ an \emph{$j$-edge} is an ordered pair $(p,q)$ of 
vertices of $D$, such that there are exactly $j$ vertices of $D$ to the left of the directed
straight line from $p$ to $q$. Let $e_j(D)$ be the number of $j$-edges of $D$. For every $0 \le k \le n-2$, let  
$E_k(S):=\sum_{j=0}^k e_j(S)$. The following equality was shown independently by Lov\'asz, Vesztergombi, Wagner and Welzl~\cite{lovaz},
and Ábrego and Fernández-Merchant~\cite{cotainferior}. 
\begin{equation}\label{eq:cr_kedges}\rcrs(D)=\sum_{k<\frac{n-2}{2}} E_k(n-2k-3)-\frac{3}{4}\binom{n}{3}+c_n \end{equation}
where 
\[c_n= \begin{cases} \frac{1}{4}E_{\frac{n-3}{2}}  &\textrm{ if } n \textrm{ is odd,} \\ 0   &\textrm{ if } n \textrm{ is even.}\end{cases}\]
Thus, lower bounds on $E_k$ provide lower bounds of $\rcrs(K_n)$. Aichholzer, Garc\'ia, Orden and Ramos~\cite{aichholzernew} showed that for every $0 \le k \le \lfloor (n-2)/2 \rfloor$, we have that
\begin{equation}\label{eq:lower_E} E_k(S) \ge 3\binom{k}{2}+\sum_{j=\lfloor n/3\rfloor}^k (3j-n+3).\end{equation}

\begin{proof}[Proposition \ref{prop:crs<rcrs}]
For $n=10,\dots,161$, the result can be verified by comparing $H(n)$ with the lower bound on $\rcrs(K_n)$ given by Equations~\ref{eq:cr_kedges} and~\ref{eq:lower_E}.
We show these values on Table~\ref{tab:values}\footnote{We point out that many of these are not the best lower bounds known; however, they are sufficient for our purposes.}.

Let $n>162$, and let $D$ be a rectilinear drawing $K_n$. For every vertex $p$ of $D$ consider the rectilinear drawing
of $K_{n-1}$ produced by removing $p$ from $D$. There are at least $\rcrs(K_{n-1})$ crossings in this drawing. 
Every crossing of $D$ is counted $n-4$ times in this way. Therefore,
\begin{align*} 
\rcrs(K_n) & \ge \frac{n}{n-4} \rcrs(K_{n-1}) \\ 
& \ge \frac{n}{n-4}\cdot \frac{n-1}{n-5} \cdot \frac{n-2}{n-6} \cdot \frac{n-3}{n-7} \cdots  \frac{162}{158} \cdot \frac{161}{157}\cdot \frac{160}{156}\cdot \frac{159}{155} \cdot\rcrs(K_{158})\\
&=n\cdot(n-1) \cdot (n-2) \cdot (n-3) \cdot \left ( \frac{1}{158} \cdot \frac{1}{157} \cdot \frac{1}{156} \cdot  \frac{1}{155} \right )\cdot\rcrs(K_{158})\\
&\ge \frac{9372519}{599809080} \cdot n\cdot(n-1) \cdot (n-2) \cdot (n-3) \cdot \left ( \frac{1}{158} \cdot \frac{1}{157} \cdot \frac{1}{156} \cdot  \frac{1}{155} \right ) \\
& = 0.015625837\cdot n\cdot(n-1) \cdot (n-2) \cdot (n-3).
\end{align*}
If $n$ is even, then
\begin{align*} 
H(n)&=\frac{1}{64} \cdot n \cdot (n-2) \cdot (n-2) \cdot (n-4) \\
 &= \frac{1}{64} \left( \frac{(n-2)(n-4)}{(n-1)(n-3)} \right)\cdot n \cdot (n-2) \cdot (n-2) \cdot (n-4)\\
 & < \frac{1}{64} \cdot n \cdot (n-2) \cdot (n-2) \cdot (n-4) \\
 &=0.015625 \cdot  n \cdot (n-2) \cdot (n-2) \cdot (n-4); \\
\end{align*}
and if $n$ is odd, then
\begin{align*} 
H(n)&=\frac{1}{64} \cdot (n-1) \cdot (n-1) \cdot (n-3) \cdot (n-3) \\
 &= \frac{1}{64} \left( \frac{(n-1)(n-3)}{n(n-2)} \right)\cdot n \cdot (n-2) \cdot (n-2) \cdot (n-4)\\
 & < \frac{1}{64} \cdot n \cdot (n-2) \cdot (n-2) \cdot (n-4) \\
 &=0.015625 \cdot  n \cdot (n-2) \cdot (n-2) \cdot (n-4). \\
\end{align*}
Therefore, \[ \rcrs(K_n) > \crs(K_n),\]
for all $n \ge 10$. 
\end{proof}
\begin{table}[]
\begin{tabular}{|c|c|c|c|c|c|c|c|c|c|c|c|}
\hline
$n$ & $H(n)$ & $cr(n) \ge $ & $n$ & $H(n)$ & $cr(n) \ge$  & $n$ & $H(n)$ & $cr(n) \ge$  & $n$ & $H(n)$ & $cr(n) \ge$\\ \hline
10 & 60 & 62 & 48 & 69828 & 70836 & 86 & 777483 & 788053 & 124 & 3460530 & 3506170  \\ \hline
11 & 100 & 101 & 49 & 76176 & 77224 & 87 & 815409 & 826182 & 125 & 3575881 & 3622541  \\ \hline
12 & 150 & 153 & 50 & 82800 & 84012 & 88 & 854238 & 865823 & 126 & 3693123 & 3741633  \\ \hline
13 & 225 & 227 & 51 & 90000 & 91212 & 89 & 894916 & 906802 & 127 & 3814209 & 3863939  \\ \hline
14 & 315 & 323 & 52 & 97500 & 98916 & 90 & 936540 & 949140 & 128 & 3937248 & 3989069  \\ \hline
15 & 441 & 444 & 53 & 105625 & 107073 & 91 & 980100 & 993099 & 129 & 4064256 & 4117056  \\ \hline
16 & 588 & 601 & 54 & 114075 & 115695 & 92 & 1024650 & 1038490 & 130 & 4193280 & 4248412  \\ \hline
17 & 784 & 794 & 55 & 123201 & 124885 & 93 & 1071225 & 1085337 & 131 & 4326400 & 4382731  \\ \hline
18 & 1008 & 1026 & 56 & 132678 & 134583 & 94 & 1118835 & 1133915 & 132 & 4461600 & 4520043  \\ \hline
19 & 1296 & 1313 & 57 & 142884 & 144804 & 95 & 1168561 & 1184025 & 133 & 4601025 & 4660887  \\ \hline
20 & 1620 & 1652 & 58 & 153468 & 155658 & 96 & 1219368 & 1235688 & 134 & 4742595 & 4804833  \\ \hline
21 & 2025 & 2049 & 59 & 164836 & 167081 & 97 & 1272384 & 1289200 & 135 & 4888521 & 4951914  \\ \hline
22 & 2475 & 2521 & 60 & 176610 & 179085 & 98 & 1326528 & 1344344 & 136 & 5036658 & 5102691  \\ \hline
23 & 3025 & 3067 & 61 & 189225 & 191795 & 99 & 1382976 & 1401144 & 137 & 5189284 & 5256714  \\ \hline
24 & 3630 & 3690 & 62 & 202275 & 205135 & 100 & 1440600 & 1459912 & 138 & 5344188 & 5414016  \\ \hline
25 & 4356 & 4416 & 63 & 216225 & 219120 & 101 & 1500625 & 1520417 & 139 & 5503716 & 5575183  \\ \hline
26 & 5148 & 5238 & 64 & 230640 & 233885 & 102 & 1561875 & 1582683 & 140 & 5665590 & 5739742  \\ \hline
27 & 6084 & 6162 & 65 & 246016 & 249346 & 103 & 1625625 & 1647041 & 141 & 5832225 & 5907729  \\ \hline
28 & 7098 & 7218 & 66 & 261888 & 265518 & 104 & 1690650 & 1713243 & 142 & 6001275 & 6079751  \\ \hline
29 & 8281 & 8397 & 67 & 278784 & 282549 & 105 & 1758276 & 1781316 & 143 & 6175225 & 6255317  \\ \hline
30 & 9555 & 9705 & 68 & 296208 & 300344 & 106 & 1827228 & 1851606 & 144 & 6351660 & 6434460  \\ \hline
31 & 11025 & 11179 & 69 & 314721 & 318921 & 107 & 1898884 & 1923853 & 145 & 6533136 & 6617816  \\ \hline
32 & 12600 & 12805 & 70 & 333795 & 338437 & 108 & 1971918 & 1998081 & 146 & 6717168 & 6804868  \\ \hline
33 & 14400 & 14592 & 71 & 354025 & 358791 & 109 & 2047761 & 2074659 & 147 & 6906384 & 6995652  \\ \hline
34 & 16320 & 16580 & 72 & 374850 & 379998 & 110 & 2125035 & 2153307 & 148 & 7098228 & 7190828  \\ \hline
35 & 18496 & 18755 & 73 & 396900 & 402232 & 111 & 2205225 & 2234052 & 149 & 7295401 & 7389857  \\ \hline
36 & 20808 & 21123 & 74 & 419580 & 425378 & 112 & 2286900 & 2317281 & 150 & 7495275 & 7592775  \\ \hline
37 & 23409 & 23735 & 75 & 443556 & 449454 & 113 & 2371600 & 2402698 & 151 & 7700625 & 7800269  \\ \hline
38 & 26163 & 26569 & 76 & 468198 & 474646 & 114 & 2457840 & 2490330 & 152 & 7908750 & 8011775  \\ \hline
39 & 29241 & 29634 & 77 & 494209 & 500829 & 115 & 2547216 & 2580585 & 153 & 8122500 & 8227332  \\ \hline
40 & 32490 & 32987 & 78 & 520923 & 528021 & 116 & 2638188 & 2673148 & 154 & 8339100 & 8447650  \\ \hline
41 & 36100 & 36602 & 79 & 549081 & 556423 & 117 & 2732409 & 2768049 & 155 & 8561476 & 8672145  \\ \hline
42 & 39900 & 40488 & 80 & 577980 & 585897 & 118 & 2828283 & 2865713 & 156 & 8786778 & 8900853  \\ \hline
43 & 44100 & 44711 & 81 & 608400 & 616464 & 119 & 2927521 & 2965811 & 157 & 9018009 & 9134515  \\ \hline
44 & 48510 & 49238 & 82 & 639600 & 648336 & 120 & 3028470 & 3068370 & 158 & 9252243 & 9372519  \\ \hline
45 & 53361 & 54081 & 83 & 672400 & 681367 & 121 & 3132900 & 3173840 & 159 & 9492561 & 9614904  \\ \hline
46 & 58443 & 59311 & 84 & 706020 & 715575 & 122 & 3239100 & 3281870 & 160 & 9735960 & 9862437  \\ \hline
47 & 64009 & 64893 & 85 & 741321 & 751191 & 123 & 3348900 & 3392490 & 161 & 9985600 & 10114482  \\ \hline
\end{tabular} \caption{The values of $H(n)$ and the lower bound of $\rcrs(K_n)$ given by Equations~\ref{eq:cr_kedges} and~\ref{eq:lower_E} for $n=10,\dots,161$} \label{tab:values}
\end{table}

We continuously use that:
\begin{itemize}
    \item if $x$ is an even integer, then $\lfloor (x-1)/2 \rfloor=(x-2)/2=x/2-1$;
    \item and if $x$ is an odd integer, then $\lfloor x/2 \rfloor=(x-1)/2$.
\end{itemize}

\begin{lemma}\label{lem:H}
    \[\frac{H(n)}{3 \binom{n}{4}}\le \frac{1}{8} \left( 1-\frac{2}{n}\right).\]
\end{lemma}
\begin{proof}
    If $n$ is even, then

 \begin{align*}
        \frac{H(n)}{3 \binom{n}{4}} & =\frac{2}{n(n-1)(n-2)(n-3)}\left \lfloor \frac{n}{2} \right \rfloor \left \lfloor \frac{n-1}{2} \right \rfloor \left \lfloor \frac{n-2}{2} \right \rfloor \left \lfloor \frac{n-3}{2} \right \rfloor \\
        &= \frac{2}{n(n-1)(n-2)(n-3)} \left (\frac{n}{2} \right ) \left ( \frac{n-2}{2} \right) \left ( \frac{n-2}{2} \right ) \left ( \frac{n-4}{2} \right) \\
        &= \frac{1}{(n-1)(n-3)}  \left ( \frac{1}{2} \right ) \left ( \frac{n-2}{2} \right ) \left ( \frac{n-4}{2} \right ) \\
        &= \frac{1}{8(n-1)(n-3)}   \left ((n-2)(n-4) \right) \\ 
        &= \frac{n-2}{8}   \left (\frac{(n-4)}{(n-1)(n-3)} \right) \\ 
        &= \frac{n-2}{8}   \left (\frac{(n-4)}{n^2-4n+3} \right )\\ 
        &= \frac{n-2}{8}   \left (\frac{(n-4)n}{(n^2-4n+3)n} \right ) \\ 
        &= \frac{n-2}{8}   \left (\frac{n^2-4n}{(n^2-4n+3)n} \right) \\ 
       & < \frac{n-2}{8} \left (\frac{1}{n}
        \right) \\ 
        & = \frac{1}{8}\left (1-\frac{2}{n} \right ).
    \end{align*}

     If $n$ is odd, then

       \begin{align*}
        \frac{H(n)}{3 \binom{n}{4}} & =\frac{2}{n(n-1)(n-2)(n-3)}\left \lfloor \frac{n}{2} \right \rfloor \left \lfloor \frac{n-1}{2} \right \rfloor \left \lfloor \frac{n-2}{2} \right \rfloor \left \lfloor \frac{n-3}{2} \right \rfloor \\
        &= \frac{2}{n(n-1)(n-2)(n-3)} \left (\frac{n-1}{2} \right ) \left ( \frac{n-1}{2} \right) \left ( \frac{n-3}{2} \right ) \left ( \frac{n-3}{2} \right) \\
        &= \frac{1}{8n} \left (  \frac{(n-1)(n-3)}{n-2} \right) \\
        &= \frac{n-2}{8n} \left (  \frac{(n-1)(n-3)}{(n-2)(n-2)} \right) \\
        &= \frac{n-2}{8n} \left (  \frac{n^2-4n+3}{n^2-4n+4} \right) \\
    &< \frac{1}{8} \left ( \frac{n-2}{n} \right) \\
            & = \frac{1}{8}\left (1-\frac{2}{n} \right ).
    \end{align*}
\end{proof}

%
%
%
\begin{proof}[Lemma~\ref{lem:Hnr}]
Let \[A:=\frac{3}{8} \binom{r}{4} \frac{n^4}{r^4}=\frac{1}{64}\cdot \frac{(r-1)(r-2)(r-3)}{r^3}n^4,\]
\[B:=r\left \lfloor \frac{n/r}{2} \right \rfloor \left \lfloor \frac{n/r-1}{2} \right \rfloor \left \lfloor \frac{n-n/r}{2} \right \rfloor \left \lfloor \frac{n-n/r-1}{2} \right \rfloor,\]
and
\[C:=\binom{r}{2}\left  \lfloor \frac{n/r}{2} \right \rfloor ^2  \left  \lfloor \frac{n/r-1}{2} \right \rfloor ^2. \]

 If $n/r$ is even, then
\begin{align*}
B & =  r\left ( \frac{n}{2r} \right ) \left (\frac{n}{2r}-1 \right ) \left (\frac{n}{2}-\frac{n}{2r} \right ) \left ( \frac{n}{2}-\frac{n}{2r}-1 \right ) \\
  & = \frac{n}{2} \left (\frac{n}{2r}-1 \right )\left( \frac{n}{2} \left ( \frac{r-1}{r}\right )\right) \left( \frac{n}{2} \left ( \frac{r-1}{r} \right)-1\right ) \\
  & = \left ( \frac{n^2}{4r} -\frac{n}{2} \right ) \left( \frac{n^2}{4}  \left ( \frac{r-1}{r}\right )^2-\frac{n}{2} \left ( \frac{r-1}{r} \right )\right ) \\
  & =\frac{n^4}{16r}\left (\frac{r-1}{r} \right)^2-\frac{n^3}{8r} \left ( \frac{r-1}{r} \right )-\frac{n^3}{8}\left ( \frac{r-1}{r}\right )^2+O(n^2);
\end{align*}

 If $n/r$ is odd, and $n$ is even, then
\begin{align*}
B & = r\left \lfloor \frac{n/r}{2} \right \rfloor \left \lfloor \frac{n/r-1}{2} \right \rfloor \left \lfloor \frac{n-n/r}{2} \right \rfloor \left \lfloor \frac{n-n/r-1}{2} \right \rfloor \\
  & = r\left ( \frac{n/r-1}{2} \right ) \left (\frac{n/r-1}{2} \right ) \left ( \frac{n-n/r-1}{2} \right ) \left ( \frac{n-n/r-1}{2} \right ) \\
  &= \frac{r}{16} \left ( \frac{n^2}{r^2} -\frac{2n}{r}+1\right ) \left ( \left ( \frac{r-1}{r} \right )^2 n^2-2 \left ( \frac{r-1}{r} \right )n+1 \right)\\
  & =\frac{n^4}{16r}\left ( \frac{r-1}{r} \right )^2-\frac{n^3}{8r} \left ( \frac{r-1}{r} \right )-\frac{n^3}{8}\left ( \frac{r-1}{r}\right )^2+O(n^2) \\
\end{align*}

If $n/r$ is odd and $n$ is odd, then
\begin{align*}
B & = r\left \lfloor \frac{n/r}{2} \right \rfloor \left \lfloor \frac{n/r-1}{2} \right \rfloor \left \lfloor \frac{n-n/r}{2} \right \rfloor \left \lfloor \frac{n-n/r-1}{2} \right \rfloor \\
  & = r\left ( \frac{n/r-1}{2} \right ) \left (\frac{n/r-1}{2} \right ) \left ( \frac{n-n/r}{2} \right ) \left ( \frac{n-n/r-2}{2} \right ) \\
  &= \frac{r}{16} \left ( \frac{n^2}{r^2} -\frac{2n}{r}+1\right ) \left ( \left ( \frac{r-1}{r} \right )^2 n^2-2 \left ( \frac{r-1}{r} \right )n\right)\\
  & =\frac{n^4}{16r}\left ( \frac{r-1}{r} \right )^2-\frac{n^3}{8r} \left ( \frac{r-1}{r} \right )-\frac{n^3}{8}\left ( \frac{r-1}{r}\right )^2+O(n^2) \\
\end{align*}

If $n/r$ is even, then
\begin{align*}
C &= \binom{r}{2}\left  \lfloor \frac{n/r}{2} \right \rfloor ^2  \left  \lfloor \frac{n/r-1}{2} \right \rfloor ^2 \\
  & \frac{r(r-1)}{2} \cdot \frac{n^2}{4 r^2} \left ( \frac{n}{2r} -1 \right )^2 \\
 &= \frac{r-1}{32r^3}n^4-\frac{r-1}{8r^2}n^3+O(n^2);
 \end{align*}
if $n/r$ is odd, then 
\begin{align*}
C &= \binom{r}{2}\left  \lfloor \frac{n/r}{2} \right \rfloor ^2  \left  \lfloor \frac{n/r-1}{2} \right \rfloor ^2 \\
  & =\frac{r(r-1)}{2} \left (\frac{n/r-1}{2} \right )^2 \left ( \frac{n/r-1}{2} \right )^2 \\
  &=\frac{r(r-1)}{32} \left ( \frac{n^4}{r^4}-\frac{4n^3}{r^3}+O(n^2) \right )\\
  &= \frac{r-1}{32r^3}n^4-\frac{r-1}{8r^2}n^3+O(n^2).
\end{align*}

Therefore,
\begin{align*}
H(n,r) & \le  A+B-C+O(n^2) \\
        &= \frac{1}{16} \left ( \frac{(r-1)(r-2)(r-3)}{4r^3}+\frac{(r-1)^2}{r^3}-\frac{r-1}{2r^3} \right )n^4 \\
        &+\frac{1}{8}\left ( -\frac{r-1}{r^2} -\frac{(r-1)^2}{r^2}+\frac{r-1}{r^2}\right ) n^3 \\
        &+ O(n^2)\\
        &= \frac{1}{16} \left ( \frac{r^3-2r+r}{4r^3}\right ) n^4 +n^3 +O(n^2) -\frac{1}{16}\left ( \frac{r-1}{r}\right )^2 2 n^3+O(n^2) \\
        & =  \frac{1}{16} \left ( \frac{r-1}{r}\right )^2 \left ( \frac{n^4}{4} -2 n^3 \right ) +O(n^2).
\end{align*}
\end{proof}

\begin{lemma}\label{lem:knr}
    If $n$ is a multiple of $r$, then 
    \[ \left ( \binom{||K_n^r||}{2}-\sum_{v\in V(K_n^r)} \binom{d(v)}{2} \right )= \frac{1}{2} \left(  \frac{r-1}{r}\right )^2 \left ( \frac{n^4}{4}-n^3\right )+O(n^2).\]
\end{lemma}
\begin{proof}
Every set in the partition has $n/r$ vertices. Thus, the number  of edges between two different
sets is equal to $n^2/r^2$. Therefore, 
\[||K_n^r||=\frac{n^2}{r^2}\binom{r}{2}=\frac{n^2}{2} \cdot \frac{r-1}{r},\]
and
\[\binom{|| K_n^r||}{2}=\frac{n^4}{8}\left (\frac{r-1}{r}\right )^2-\frac{n^2}{4}\left (\frac{r-1}{r}\right ).\]
For every vertex $v$ of $K_n^r$, it holds that 
\[d(v)=\frac{r-1}{r}n.\]
Thus, 
\[\sum_{v \in V(K_n^r)} \binom{d(v)}{2}=\frac{n^3}{2}\left ( \frac{r-1}{r} \right )^2-\frac{r-1}{r} \cdot \frac{n^2}{2}.\]
It follows that 
\begin{align*}
\left ( \binom{||K_n^r||}{2}-\sum_{v\in V(K_n^r)} \binom{d(v)}{2} \right )   & = \frac{n^4}{8}\left (\frac{r-1}{r}\right )^2-\frac{n^2}{4}\left (\frac{r-1}{r}\right )- \frac{n^3}{2}\left ( \frac{r-1}{r} \right )^2+\frac{n^2}{2} \left (\frac{r-1}{r} \right ) \\
                  &=  \frac{1}{2} \left (\frac{r-1}{r} \right )\left ( \frac{n^4}{4} \left ( \frac{r-1}{r} \right )-\frac{n^2}{2}-n^3 \left (\frac{r-1}{r} \right )+n^2 \right )\\
                & =\frac{1}{2} \left (\frac{r-1}{r} \right )^2 \left(\frac{n^4}{4}-n^3 \right) +O(n^2).
\end{align*}
\end{proof}

Using Equation~\ref{eq:random_embeding},and Lemmas~\ref{lem:H} and~\ref{lem:knr}, we can prove Theorem~\ref{thm:Knr_gen}.

\begin{proof}[Theorem~\ref{thm:Knr_gen}]
By Equation~\ref{eq:random_embeding}, it holds that

\[E(\crs(D))  = \frac{H(n)}{3\binom{n}{4}} \left ( \binom{||K_n^r||}{2}-\sum_{v\in V(K_n^r)} \binom{d(v)}{2} \right ).\]
Applying Lemmas~\ref{lem:H} and~\ref{lem:knr} on the equality above yields 
\begin{align*}
              E(\crs(D))  & \le \frac{1}{8} \left ( 1-\frac{2}{n}\right )  \left (  \frac{1}{2} \left(  \frac{r-1}{r}\right )^2 \left ( \frac{n^4}{4}-n^3\right )+O(n^2) \right ) \\
                & \le \frac{1}{16} \left (\frac{r-1}{r} \right )^2 \left(\frac{n^4}{4}-\frac{3n^3}{2} \right) +O(n^2).
\end{align*}
\end{proof}

\begin{proof}[Theorem~\ref{lem:upper_rect}]
From Equation~\ref{eq:random_embeding} and the best upper bound known for $\rcrs(K_n)$, it follows that 
\begin{align*}
    E(\rcrs(\overline{D}))  &= \frac{\rcrs(K_n)}{3\binom{n}{4}} \left ( \binom{||K_n^r||}{2}-\sum_{v\in V(K_n^r)} \binom{d(v)}{2} \right )\\
                & = \frac{\overline{q} \binom{n}{4}+o(n^4)}{3\binom{n}{4}} \left (  \frac{1}{2} \left(  \frac{r-1}{r}\right )^2 \left ( \frac{n^4}{4}-n^3\right )+O(n^2) \right ) \\
                & \le \frac{\overline{q}}{4!} \left( \frac{r-1}{r} \right )^2 n^4+o(n^4).
\end{align*}
\end{proof}


\begin{proof}[Theorem~\ref{lem:lower_rect}]
Let $D$ be a rectilinear drawing of $K_n^r$. Let $D'$ be a rectilinear drawing of $K_r$ obtained
by choosing one point from each color class of $D$. There are $(n/r)^r$ such choices; and each choice
provides at least $\rcrs(K_r)$ crossings. Each  such crossing is counted exactly $(n/r)^{r-4}$ times.
\end{proof}

\begin{proof}[Corollary~\ref{cor:lime}]
We have that 
\begin{align*} 
\lim_{n \to \infty} \frac{\rcrs(K_n^r)}{\binom{n}{4}}  &\ge \lim_{n \to \infty} \rcrs(K_r) \cdot \left(  \frac{n}{r}\right )^4\cdot \frac{4!}{n(n-1)(n-2)(n-3)} \\
& \ge \lim_{r \to \infty} \frac{\rcrs(K_r) }{(\frac{r^4}{4!})} \\
& =\overline {q}.
\end{align*}
By Theorem~\ref{lem:upper_rect}, $\lim_{n \to \infty} \frac{\rcrs(K_n^r)}{\binom{n}{4}} \le \overline{q} \left( \frac{r-1}{r} \right )^2.$ 

As $\left( \frac{r-1}{r} \right )^2 < 1$, it follows that
\begin{align*} 
\lim_{n \to \infty} \frac{\rcrs(K_n^r)}{\binom{n}{4}}  = \overline {q}.
\end{align*}
\end{proof}

To prove Theorem~\ref{thm:layers_random}, we use the following proposition.
\begin{prop}\label{prop:layer}
    Let $r$ be a positive integer and let $n$ be a multiple of $r$.
    Then \[\left ( \binom{||L_n^r||}{2}-  \sum_{v \in V(L_n^r)} \binom{d(v)}{2} \right ) =\frac{(r-1)^2}{2r^4}n^4-\frac{2r-3}{r^3}n^3+\frac{r-1}{r^2}n^2.\]
\end{prop}
\begin{proof}
Note that \[||L_n^r||=(r-1)\left ( \frac{n}{r} \right )^2;\]
and
\[\binom{||L_n^r||}{2}=\frac{(r-1)^2}{2r^4}n^4-\frac{r-1}{2r^2}n^2.\]
We have that
\begin{align*}
    \sum_{v \in V(L_n^r)} \binom{d(v)}{2} & =\frac{2n}{r} \binom{n/r}{2}+\frac{(r-2)n}{r}\binom{2n/r}{2} \\
    & = \frac{2n}{r} \left ( \frac{n^2}{2r^2}-\frac{n}{2r}\right ) +\frac{(r-2)n}{r} \left ( \frac{2n^2}{r^2}-\frac{n}{r}\right ) \\
    & =\frac{2r-3}{r^3}n^3-\frac{r-1}{r^2}n^2.
\end{align*}
Thus,
\[\left ( \binom{||L_n^r||}{2}-  \sum_{v \in V(L_n^r)} \binom{d(v)}{2} \right ) =\frac{(r-1)^2}{2r^4}n^4-\frac{2r-3}{r^3}n^3+\frac{r-1}{r^2}n^2.\]
\end{proof} 

Combining Proposition~\ref{prop:layer}, Equation~\ref{eq:random_embeding} and Lemma~\ref{lem:H}, we obtain Theorem~\ref{thm:layers_random}.
\begin{proof}[Theorem~\ref{thm:layers_random}]
\begin{align*}
   \crs(L_n^r) \le &E(\crs(D)) = \frac{H(n)}{3\binom{n}{4}} \left ( \binom{||L_n^r||}{2}-\sum_{v\in V(L_n^r)} \binom{d(v)}{2} \right ) \le\\
   &\frac{1}{8} \left( 1-\frac{2}{n}\right)\left(\frac{(r-1)^2}{2r^4}n^4+O(n^3)\right)\le \frac{(r-1)^2}{16 r^4}n^4+O(n^3).
   \end{align*}
\end{proof}

\begin{proof}[Lemma~\ref{lem:seed}]
We classify the crossings of $D^s$ depending on the number of different clusters in which the endpoints of the edges defining the crossing appear. Let $e_1$ and $e_2$ be a pair of edges of $D^s$ that cross.

Suppose that the endpoints of $e_1$ and $e_2$ appear in four different clusters.
We have that $e_1=(u,i)(v,j)$ and $e_2=(w,k)(x,l)$ for some four distinct vertices
$u,v,w,x$ of $D$ and indices $1 \le i,j,k,l \le s$. Thus, $uv,wx$ is a pair
of crossing edges in $D$; and for each pair of crossing edges in $D$ we obtain
$s^4$ pairs of crossing edges of $D^s$, such that its endpoints lie in four different clusters. Therefore, the number of crossings of $D^s$ generated by pairs of edges whose endpoints lie in four different clusters is equal to
\[\crs(D)s^4.\]

Suppose that the endpoints of $e_1$ and $e_2$ lie in three different clusters.
Without loss of generality $e_1=(u,i)(v,j)$ and $e_2=(u,k)(w,l)$ for some three distinct
vertices $u,v,w$ of $D$ and indices $1 \le i,j,k,l \le s$. Note that $v$ and $w$ lie on the same side of  $\ell_u$, otherwise
 $\ell_u$ separates the edge $uv$ from the edge $uw$ and no crossing between $e_1$ and $e_2$ would be possible. 
Conversly, for every pair of vertices of $D$ lying on the same side of $\ell_u$ we obtain $\binom{s}{2}s^2$ crossings
in $D^s$ generated by pairs of edges whose endpoints lie in three different clusters. Therefore,
the number of crossings of $D^s$ generated by pairs of edges whose endpoints lie in three different clusters
is equal to
\[\sum_{v \in V(G)} \left ( \binom{ \lfloor d(v)/2 \rfloor}{2}+ \binom{\lceil d(v)/2 \rceil}{2}\right)\frac{s^3(s-1)}{2}.\]

Suppose that the endpoints of $e_1$ and $e_2$ lie in two different clusters.
We have that $e_1=(u,i)(v,j)$ and $e_2=(u,k)(v,l)$ for some edge $uv$ of $D$
and indices $1 \le i,j,k,l \le s$; and for every edge of $D$ we obtain $\binom{s}{2}\binom{s}{2}$ crossings
in $D^s$ generated by pairs of edges whose endpoints lie in two different clusters.
Therefore,
the number of crossings of $D^s$ generated by pairs of edges whose endpoints lie in two different clusters
is equal to
\[||G||\frac{s^2(s-1)^2}{4}.\]
\end{proof}

We now give the coordinates of the rectilinear drawing $D$ of $K_{24}^4$ with 2033 crossings. The colors are $0,1,2$ and $3$.
We have appended the color of each point as a third coordinate.
\begin{align*}
V(D)=\{&(-59260959, 44970123, 0),
 (261261347, -43693014, 0),
 (158829052, -28658158, 0),\\
&(-20273112, -23913465, 0),
 (20602644, -8343316, 0),
 (-8148611, -63519416, 0), \\
 & (30209164, 4850528, 1),
 (12317574, -161508817, 1),
 (46649346, -344926319, 1),\\
&(-11015825, -47872739, 1),
 (-26347789, 22655563, 1),
 (-46729617, 35472331, 1),\\
 & (-74136586, 66127255, 2),
 (-278900322, 316137789, 2),
 (14791528, -20163276, 2),\\
 &(-140757971, 147565111, 2),
 (14081248, -20874215, 2),
 (9903931, -24183515, 2), \\
&(-38516867, 27953341, 3),
 (-60922797, 47350463, 3),
 (8267623, -135305393, 3),\\
 &(-15043716, -39580158, 3),
 (41831995, 797354, 3),
 (181333931, -34086725, 3)\}.
 \end{align*}
 The vertices of this drawing can be seen in Figure~\ref{fig:k24}.
 \begin{figure}
	\centering
	\includegraphics[width=1.0\linewidth]{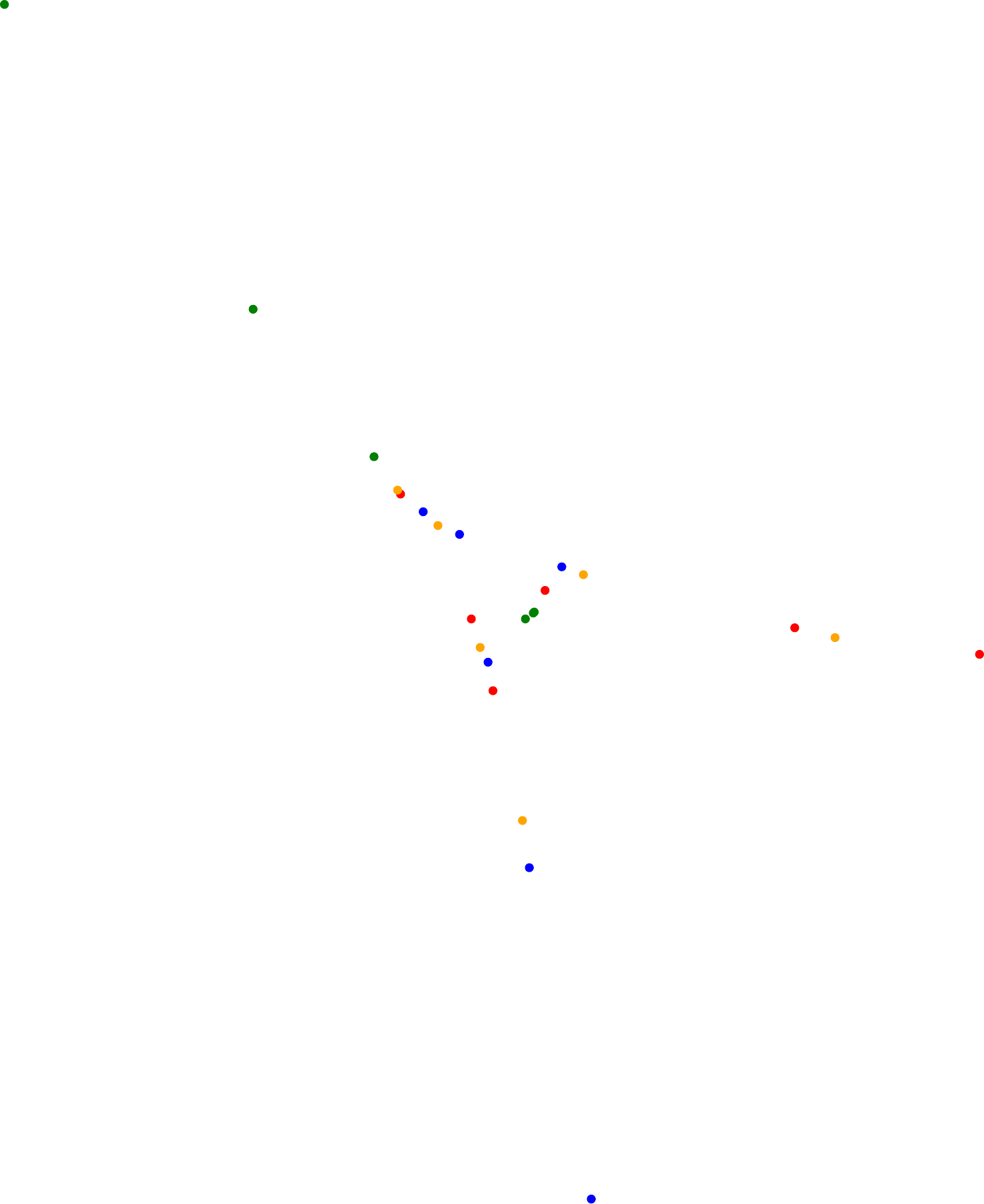}
	\caption{The vertices of a rectilinear drawing of $K_{24}^4$}
	\label{fig:k24}
\end{figure}
\end{document}